\documentclass[a4paper,10pt]{amsart}
\usepackage{enumerate}
\usepackage{amsthm,amsmath,amssymb,graphics,graphicx,hyperref,epstopdf}
\usepackage{breqn}           
\title{Minimal Resolutions of Dominant and Semidominant Ideals}
\author{Guillermo C. Alesandroni}
\address{Department of Mathematics, Oklahoma State University, 401 Mathematical Sciences, Stillwater, OK 74078}
\email{guillea@okstate.edu}

\newtheorem{theorem}{Theorem}[section]
\newtheorem{proposition}[theorem]{Proposition}
\newtheorem{corollary}[theorem]{Corollary}
\newtheorem{lemma}[theorem]{Lemma}

\theoremstyle{definition}
\newtheorem{definition}[theorem]{Definition}
\newtheorem{remark}[theorem]{Remark}
\newtheorem{example}[theorem]{Example}

\DeclareMathOperator{\btwo}{b_2}
\DeclareMathOperator{\bone}{b_1}
\DeclareMathOperator{\bzero}{b_0}
\DeclareMathOperator{\bi}{b_i}
\DeclareMathOperator{\pd}{pd}
\DeclareMathOperator{\mdeg}{mdeg}
\DeclareMathOperator{\lcm}{lcm}
\DeclareMathOperator{\Ker}{Ker}
\DeclareMathOperator{\im}{Im}

\DeclareMathOperator{\reg}{reg}
\DeclareMathOperator{\hdeg}{hdeg}

\begin{document}
\maketitle
\begin{abstract}
 We construct the minimal resolutions of three classes of monomial ideals: dominant, 1-semidominant, and 2-semidominant ideals. The families of dominant and
 1-semidominant ideals extend those of complete and almost complete intersections. 
 We show that dominant ideals give a precise characterization of when the Taylor resolution is minimal, 1-semidominant ideals are Scarf, and the minimal 
 resolutions of 2-semidominant ideals can be obtained from their Taylor resolutions by eliminating faces and facets of equal multidegree, in 
 arbitrary order. We study the combinatorial properties of these classes of ideals and explain how they relate to generic ideals.
\end{abstract}
\section{Introduction}
In this paper we explore the minimal free resolutions of dominant, 1-semidominant and 2-semidominant ideals, three families of monomial ideals that are easy to describe and 
have strong combinatorial properties.

For over half a century mathematicians have tried to obtain the minimal resolutions of families of ideals in closed form with little success. A common mark
in the construction of these classes of ideals and their corresponding resolutions has been the use of a monomial ordering or, at least, an ordering of the
variables. Groebner bases, mapping cones, Borel ideals and the (usually nonminimal) Lyubeznik resolution [No,Pe,Me]  are some examples of this phenomenon.

Dominant, 1-semidominant and 2-semidominant ideals, as well as the technique that resolves them minimally, distinguish from the objects mentioned above in 
that they do not
require an ordering of the variables; instead, they are characterized by the exponents with which the variables appear in the factorization of the monomial
generators. The concept of dominance resembles the definition of generic ideal [BPS,BS] as we will explain in section 4.

We will show that the minimal free resolutions of these classes of ideals have some important properties. 
In particular, the Taylor resolution of a monomial ideal is minimal if and only if the ideal is dominant. 
In other words, dominant ideals give a full and explicit characterization of when the Taylor resolution is minimal.

The minimal resolutions of 1-semidominant ideals are also remarkably simple; they are given by the Scarf complex. 
Thus it would be fair to say that we know everything about them. Although not as easy to decode as in the first two cases, the minimal 
resolutions of 2-semidominant ideals can also be expressed in simple terms: informally speaking, they can be obtained from their Taylor resolutions
eliminating pairs of face and facet of equal multidegree in arbitrary order, until exhausting all possibilities.

The concepts of dominant and 1-semidominant ideal extend those of complete and almost complete intersection in a natural way, and the transition from 
dominant to 1-semidominant ideal is smooth. The latter definition is obtained from the former via a minor modification. However, the combinatorial 
properties of dominant and 1-semidominant ideals can be radically different. For instance, in section 5 we give a condition under which a dominant ideal 
and a 1-semidominant ideal (that look almost identical) have the largest and smallest possible projective dimensions, respectively.

The structure of this paper is as follows. In section 2, we adopt a few conventions, fix the notation that will be used frequently, and 
describe some background material. Section 3 is technical. There we develop the machinary that will be instrumental in the proofs of the results announced
above.

Sections 4, 5, and 6 deal with dominant, 1-semidominant, and 2-semidominant ideals, respectively. In particular, at the end of section 6 we describe 
all 2-semidominant ideals that are Scarf. In section 7, we explain how $p$-semidominant ideals, a generalization of 1- and 2-semidominant ideals, increase in 
complexity as the value of $p$ grows. We also discuss conditions under which the minimal resolution of a monomial ideal can be obtained from its Taylor 
resolution by eliminating faces and facets of equal multidegree in arbitrary order.

\section{Background and Notation}
We will assume that the reader is familiar with the concepts of simplicial resolution, Taylor resolution, Scarf complex and Scarf ideal. A good source to 
learn these prerequisites is [Me]. It would also be helpful (although not essential)
to be acquainted with the following suplementary material: generic ideals, complete intersections, Betti numbers, projective dimension and regularity. The 
canonical place to learn these topics is [Pe].

Regarding notation, it is convenient to fix some terminology to identify objects that will be used repeatedly. 
Throughout the paper $S$ represents a polynomial ring over an arbitrary field, in a finite number variables. The letter $M$ always denotes a monomial ideal
in $S$, while the symbol $\mathbb{T}_M$ is used to identify the Taylor resolution of $S/M$.

In addition to the nomenclature fixed above, we need to set the following convention. Following [Me], the Taylor resolution $\mathbb{T}_M$ will be viewed as a simplicial 
resolution rather than the more 
classical construction based on the exterior algebra over a field. With this interpretation, if $\Delta_M$ is the full simplex on $M$ we can identify
its sets $\{l_1,\ldots,l_p\}$ with the basis elements $[l_1,\ldots, l_p]$ of $\mathbb{T}_M$. Based on this correspondence and abusing the language, basis
elements of the form $[l_1,\ldots, l_p]$ and $[\l_1,\ldots,\widehat{l_i},\ldots, l_p]$ will be referred to as being face and facet. We will also say that 
$l_1,\ldots,l_p$ are contained in $[l_1,\ldots, l_p]$. The multidegrees of the basis elements $[\sigma]$ of $\mathbb{T}_M$, denoted $\text{mdeg}[\sigma]$, 
will be written as monomials. More precisely, if $[\sigma]=[l_1,\ldots, l_p]$, then $\mdeg[\sigma]=\lcm(l_1,\ldots,l_p)$.

Finally, whenever we mention the $j^{th}$ differential map of $\mathbb{T}_M$, we make reference to the map with domain the free module in homological
degree $j$. In other words, if $M=(l_1,\ldots,l_q)$, and
\[\mathbb{T}_M:0\rightarrow F_q\xrightarrow{f_q}\cdots\rightarrow F_j\xrightarrow{f_j}F_{j-1}\xrightarrow{f_{j-1}}\cdots\rightarrow F_0\xrightarrow{f_0}
 S/M\rightarrow0,\]
 the $j^{th}$ differential map of $\mathbb{T}_M$ is $f_j:F_j\rightarrow F_{j-1}$, where \[f_j\left([l_{r_1},\ldots, l_{r_j}]\right)=
 \sum_{i=1}^j(-1)^{i+1}\dfrac{\mdeg [l_{r_1},\ldots, l_{r_j}]}{\mdeg [l_{r_1},\ldots,\widehat{l_{r_i}},\ldots, l_{r_j}]}
 [l_{r_1},\ldots,\widehat{l_{r_i}},\ldots, l_{r_j}]\]
\section{Foundational Results}
The results in this section are foundational in character because they deal with the basic concepts of change of basis and consecutive cancellation, which
are a natural avenue leading to the minimal free resolution of a monomial ideal. Most of these results are known in some form to experts, yet we have decided
to include statements with full proofs because the material is essential to the development of this paper and, as far as we know, nobody has published these
particular facts with careful explanations.

The reader will find that the underlying ideas have the strong familiar flavor of linear algebra.
\begin{definition}
 Let $M$ be a monomial ideal and let 
\[0\rightarrow F_q\xrightarrow{f_q}\cdots\rightarrow F_{j+2}\xrightarrow{f_{j+2}}F_{j+1}\xrightarrow{f_{j+1}}F_j\xrightarrow{f_j}F_{j-1}\rightarrow\cdots
\rightarrow F_0\rightarrow S/M\rightarrow0\] 
be a free resolution of $S/M$.\\
Let $U=\{[u_1],\cdots,[u_h]\}$ be a basis of $F_{j+1}$ and let $V=\{[v_1],\cdots,[v_g]\}$ be a basis of $F_j$. Suppose $a_{rs}$ is an invertible entry of the
differential matrix
\[(f_{j+1})_{U,V}=\left(\begin{array}{ccccc}
a_{11}&\cdots&a_{1s}&\cdots&a_{1h}\\
\vdots&      &\vdots&      &\vdots\\
a_{r1}&\cdots&a_{rs}&\cdots&a_{rh}\\
\vdots&      &\vdots&      &\vdots\\
a_{g1}&\cdots&a_{gs}&\cdots&a_{gh}
\end{array}\right)\]
The change of basis $U'=\{[u_1]',\cdots,[u_h]'\}$, where $[u_s]'=[u_s]$ and $[u_i]'=[u_i]-\dfrac{a_{ri}}{a_{rs}}[u_s]$ for all $i\neq s$;
and $V'=\{[v_1]',\ldots,[v_g]'\}$, where $[v_r]'=\sum\limits_{i=1}^ga_{is}[v_i]$ and $[v_i]'=[v_i]$, for all $i\neq r$
 will be called \textbf{the standard change of basis} (around $a_{rs}$).
\end{definition}
\begin{lemma}
With the notation used in Definition 3.1, if we make a standard change of basis around $a_{rs}$, the following properties hold:
\begin{enumerate}[(i)]
\item $\mdeg[u_i]'=\mdeg[u_i]$, for all $i=1,...,h$; $\mdeg[v_i]'=\mdeg[v_i]$, for all $i=1,...,g$.
\item The differential matrix $(f_{j+1})_{U',V'}$ is of the form
\[(f_{j+1})_{U',V'}=\left(\begin{array}{ccccccc}
b_{1,1}&...&b_{1,s-1}&0&b_{1,s+1}&...&b_{1,h}\\
\vdots&   &\vdots&\vdots&\vdots&  &\vdots\\
b_{r-1,1}&...&b_{r-1,s-1}&0&b_{r-1,s+1}&...&b_{r-1,h}\\
0      &...&0          &1&0         &...&0\\
b_{r+1,1}&...&b_{r+1,s-1}&0&b_{r+1,s+1}&...&b_{r+1,h}\\
\vdots&     &\vdots&\vdots&         &   &\vdots\\
b_{g,1}&...&b_{g,s-1}&0&b_{g,s+1}&...&b_{g,h}
\end{array}\right)\]
\item Let $1\leq c\leq g$ and $1\leq d\leq h$. If $c\neq r$ and $d\neq s$, then $b_{cd}=a_{cd}-\dfrac{a_{rd}a_{cs}}{a_{rs}}$.
\item The differential matrix $(f_{j+2})_{T,U'}$ is obtained from $(f_{j+2})_{T,U}$ by turning the $s^{th}$ row into a row of zeros, and the differential
matrix $(f_j)_{V',W}$ is obtained from $(f_j)_{V,W}$ by turning the $r^{th}$ column into a column of zeros. (Here we assume that $T$ and $W$ are bases
of $F_{j+2}$ and $F_{j-1}$, respectively.)
\end{enumerate}
 \end{lemma}
\begin{proof}
 \item (i) This part is essentially a consequence of the fact that $f_{j+1}$ is a graded map of degree $0$.\\
First, notice that since $f_{j+1}([u_s])=\sum\limits_{i=1}^ga_{is}[v_i]$, we must have 
\[\mdeg[u_s]=\mdeg(a_{is}[v_i])=\mdeg a_{is}\mdeg[v_i]\text{, for all }i\]
In particular, since $\mdeg a_{rs}=1$, we have that $\mdeg[u_s]=\mdeg[v_r]$. On the other hand, 
\[\mdeg[u_s]=\mdeg(f_{j+1}([u_s])=\mdeg\left(\sum\limits_{i=1}^ga_{is}[v_i]\right)=\mdeg[v_r]'.\]
Combining these facts, we get that $\mdeg[v_r]'=\mdeg[v_r]$. In addition to this, it is clear that for all $i\neq r$, $\mdeg[v_i]'=\mdeg[v_i]$, 
 which proves the first part of (i).
 
Now given that $f_{j+1}([u_i])=\sum\limits_{p=1}^ga_{pi}[v_p]$, we must have that $\mdeg[u_i]=\mdeg\left(a_{pi}[v_p]\right)$, for all $i=1,...,h$ and $p=1,...,g$.
In particular, $\mdeg[u_i]=\mdeg\left(a_{ri}[v_r]\right)$. Therefore, \\
$\mdeg\left(\dfrac{a_{ri}}{a_{rs}}[u_s]\right)=\mdeg\left(\dfrac{a_{ri}}{a_{rs}}\right)\mdeg[u_s]=\mdeg a_{ri}\mdeg[v_r]=
\mdeg(a_{ri}[v_r])=\mdeg[u_i]$, which shows that $[u_i]'=[u_i]-\dfrac{a_{ri}}{a_{rs}}[u_s]$ is homogeneous and 
$\mdeg[u_i]'=\mdeg[u_i]$.
Finally, it is clear that $\mdeg[u_s]'=\mdeg[u_s]$.
\item (ii) $f_{j+1}([u_s]')=f_{j+1}([u_s])=\sum\limits_{i=1}^ga_{is}[v_i]=[v_r]'$. Therefore, the $s^{th}$ column of $(f_{j+1})_{U',V'}$ is as stated in the 
lemma.

On the other hand, for all $i\neq s$,
\begin{dmath*}
f_{j+1}([u_i]')=f_{j+1}\left([u_i]-\dfrac{a_{ri}}{a_{rs}}[u_s]\right)=f_{j+1}([u_i])-\dfrac{a_{ri}}{a_{rs}}f_{j+1}([u_s])=\sum\limits_{p=1}^ga_{pi}[v_p]-
\dfrac{a_{ri}}{a_{rs}}\sum\limits_{p=1}^ga_{ps}[v_p]=\sum\limits_{p\neq r}\left(a_{pi}-\dfrac{a_{ri}}{a_{rs}}a_{ps}\right)[v_p]+0[v_r]=\sum\limits_{p\neq r}
\left(a_{pi}-\dfrac{a_{ri}}{a_{rs}}a_{ps}\right)[v_p]'+0[v_r]'
\end{dmath*}
Hence, the $r^{th}$ row of $(f_{j+1})_{U',V'}$ is as stated.
\item (iii) If $c\neq r$ and $d\neq s$, we have
\begin{dmath*}
f_{j+1}([u_d]')=f_{j+1}\left([u_d]-\dfrac{a_{rd}}{a_{rs}}[u_s]\right)=\sum\limits_{i=1}^ga_{id}[v_i]-\dfrac{a_{rd}}{a_{rs}}\sum\limits_{i=1}^ga_{is}[v_i]=
\sum\limits_{i\neq c \ i\neq r }\left(a_{id}-\dfrac{a_{rd}}{a_{rs}}a_{is}\right)[v_i]+\left(a_{cd}-\dfrac{a_{rd}}{a_{rs}}a_{cs}
\right)[v_c]+0[v_r]=
\sum\limits_{i\neq c \ i\neq r}\left(a_{id}-\dfrac{a_{rd}}{a_{rs}}a_{is}\right)[v_i]'+\left(a_{cd}-
\dfrac{a_{rd}}{a_{rs}}a_{cs}\right)[v_c]'
\end{dmath*}
This implies that $b_{cd}=a_{cd}-\dfrac{a_{rd}a_{cs}}{a_{rs}}$.
\item (iv) We will denote with $A_{ip}$ the entries of $\left(f_{j+2}\right)_{T,U}$ and with $B_{ip}$ the entries of $\left(f_{j+2}\right)_{T,U'}$.
If $[t_p]$ is a basis element in $T$, $f_{j+2}([t_p])=\sum\limits_{i=1}^hA_{ip}.[u_i]$.\\
Given that for all $i\neq s$,  $[u_i]=[u_i]'+\dfrac{a_{ri}}{a_{rs}}[u_s]'$, it follows that
\begin{dmath*}
f_{j+2}([t_p])=\sum_{i\neq s} A_{ip}\left([u_i]'+\dfrac{a_{ri}}{a_{rs}}[u_s]'\right)+A_{sp}[u_s]'=
\sum_{i=1}^hA_{ip}[u_i]'+\left[\left(\sum_{i\neq s} A_{ip}\dfrac{a_{ri}}{a_{rs}}\right)+A_{sp}\right][u_s]'.
\end{dmath*}
This implies that, for all $i\neq s$, $B_{ip}=A_{ip}$. 

On the other hand, the entry $B_{sp}=\left(\sum\limits_{i\neq s} A_{ip}\dfrac{a_{ri}}{a_{rs}}\right)+A_{sp}$ must be zero, 
as we show below.\\
Since $\im f_{j+2}=\Ker f_{j+1}$, we must have $f_{j+1}\circ f_{j+2}([t_p])=0$; that is\\
\begin{dmath*}
\left(\begin{array}{c}0\\\vdots\\0\end{array}\right)=\left(f_{j+1}\right)_{U',V'}\left(f_{j+2}\right)_{T,U'}([t_p])=
\left(\begin{array}{ccccccc}
b_{1,1}&...&b_{1,s-1}&0&b_{1,s+1}&...&b_{1,h}\\
\vdots&   &\vdots&\vdots&\vdots&  &\vdots\\
b_{r-1,1}&...&b_{r-1,s-1}&0&b_{r-1,s+1}&...&b_{r-1,h}\\
0      &...&0          &1&0         &...&0\\
b_{r+1,1}&...&b_{r+1,s-1}&0&b_{r+1,s+1}&...&b_{r+1,h}\\
\vdots&     &\vdots&\vdots&         &   &\vdots\\
b_{g,1}&...&b_{g,s-1}&0&b_{g,s+1}&...&b_{g,h}
\end{array}\right)\left(\begin{array}{c}A_{1p}\\\vdots\\A_{s-1p}\\
\left(\sum\limits_{i\neq s} A_{ip}\dfrac{a_{ri}}{a_{rs}}\right)+A_{sp}\\
A_{s+1p}\\ \vdots \\A_{hp}\end{array}\right)
\end{dmath*}
Notice that the $s^{th}$ entry of the resulting column vector is 
$0=\left(\sum\limits_{i\neq s} A_{ip}\dfrac{a_{ri}}{a_{rs}}\right)+A_{sp}$.\\
This proves our statement regarding 
$\left(f_{j+2}\right)_{T,U'}$.

The proof of the second statement is as follows: for all $i\neq r$, $[v_i]'=[v_i]$, which means that $f_j([v_i]')=f_j([v_i])$. In turn, this implies that
the $i^{th}$ columns of $(f_j)_{V',W}$ and $(f_j)_{V,W}$ are equal. Finally, since $[v_r]'=f_{j+1}([u_s]')\subseteq \im f_{j+1}=\Ker f_j$,
we must have $f_j([v_r]')=0$, which means that the $r^{th}$ column of $(f_j)_{V',W}$ is a column of zeros, as stated.
\end{proof}

Lemma 3.2 has several important implications that we discuss next. We continue to use the notation introduced in that lemma.
\begin{remark}
 It is obvious that when we make a standard change of basis, some of the basis elements $[u_i]$ and $[v_i]$ change. However, since the free 
modules $S[u_i]$ and $S[u_i]'$ (respectively $S[v_i]$ and $S[v_i]'$) are isomorphic, and given that by Lemma 3.2 (i), $[u_i]$ and $[u_i]'$ (respectively
$[v_i]$ and $[v_i]'$) are abstract objects with the same multidegree, we can assume that the basis elements $[u_i]$ and $[v_i]$ do not change. 
 Therefore, after making a standard change of basis, we can interpret that we have two different representations
 \[\cdots\rightarrow\bigoplus S[u_i]\xrightarrow{\left(f\right)}\bigoplus S[v_i]\rightarrow\cdots\] and
 \[\cdots\rightarrow\bigoplus S[u_i]'\xrightarrow{\left(f\right)'}\bigoplus S[v_i]'\rightarrow\cdots\]
 of the same free resolution of $S/M$, or we can interpret that we have two representations
 \[\cdots\rightarrow\bigoplus S[u_i]\xrightarrow{\left(f\right)}\bigoplus S[v_i]\rightarrow\cdots\] and
 \[\cdots\rightarrow\bigoplus S[u_i]\xrightarrow{\left(f\right)'}\bigoplus S[v_i]\rightarrow\cdots\]
 of two different free resolutions of $S/M$. We will choose the second interpretation. This way, if we identify the basis of $\mathbb{T}_M$ with a simplicial
 complex, when we make a standard change of basis or a consecutive cancellation, the basis of the new resolution can be identified with a subset of the
 simplicial complex and we can still speak in terms of faces and facets.
\end{remark}
  
 \begin{remark}
  In the same fashion that we identified the differential map $f_{j+1}$ with the differential matrix $\left(f_{j+1}\right)_{U,V}=\left(a_{rs}\right)$, 
 we can identify the 
$s^{th}$ basis element $[u_s]$ of $F_{j+1}$ with the column vector $\left(\delta_{is}\right)$, where $\delta_{is}=0$ if $i\neq s$, and $\delta_{ss}=1$.
Similarly, the image $f_{j+1}\left([u_s]\right)=\sum_{i=1}^ga_{is}[v_i]$ of $[u_s]$ can be identified with the $s^{th}$ column vector 
$\left(f_{j+1}\right)_{U,V}.\left(\delta_{is}\right)=\left(a_{is}\right)$ of $\left(a_{rs}\right)$. Thus each entry $a_{rs}$ is the coefficient of $[v_r]$ when
$f_{j+1}\left([u_s]\right)$ is expressed in terms of the basis $V=\left\{[v_1],\ldots,[v_g]\right\}$. Notice that there is a bijective correspondence
between the entries $a_{rs}$ of $\left(f_{j+1}\right)_{U,V}$ and the ordered pairs $\left([u_s],[v_r]\right)$ of basis elements $[u_s]$ and $[v_r]$ in
homological degrees $j+1$ and $j$, respectively. This means that the entry $a_{rs}$ of $\left(f_{j+1}\right)_{U,V}$ can be written $a_{\tau\sigma}$, where
$[\sigma]$ is the $s^{th}$ basis element of $U$ and $[\tau]$ is the $r^{th}$ basis element of $V$. That is, instead of using subscripts that denote the 
number of row and column where the entry is placed, we can use subscripts that identify the basis elements that generate this entry. Most of the time we will
choose the notation $a_{\tau\sigma}$ over $a_{rs}$ and will say that $a_{\tau\sigma}$ is determined by $[\sigma]$ and $[\tau]$.
 \end{remark}
\begin{remark}
 Since $f_{j+1}$ is graded of degree $0$, if $a_{rs}\neq0$ we must have 
\[\mdeg[u_s]=\mdeg f_{j+1}\left([u_s]\right)=\mdeg\left(\sum_{i=1}^ga_{is}[v_i]\right)=\mdeg\left(a_{rs}[v_r]\right)=
\mdeg a_{rs}\mdeg[v_r]\]
Hence, $a_{rs}=0$ or $\mdeg a_{rs}=\dfrac{\mdeg[u_s]}{\mdeg[v_r]}$.\\
With the notation introduced in Remark 3.4:
$a_{\tau\sigma}=0$ or $\mdeg a_{\tau\sigma}=\dfrac{\mdeg[\sigma]}{\mdeg[\tau]}$.
In particular, if $a_{\tau\sigma}$ is invertible then $\mdeg[\sigma]=\mdeg[\tau]$.\\
Now let $b_{\tau\sigma}$ be the entry determined by $[\sigma]$ and $[\tau]$ in $\left(f_{j+1}\right)_{U',V'}$. Reasoning as before, we get
$b_{\tau\sigma}=0$ or $\mdeg b_{\tau\sigma}=\dfrac{\mdeg[\sigma]}{\mdeg[\tau]}$\\
(Informally speaking, the multidegrees of the entries do not change under standard changes of bases.) In particular, if $a_{\tau\sigma}$ is invertible then,
$b_{\tau\sigma}=0$ or $b_{\tau\sigma}$ is also invertible.
\end{remark}
 \begin{remark}
  It follows from Lemma 3.2 (ii) and (iv) that after making a standard change of basis around $a_{rs}$, it is possible to make the consecutive cancellation
$0\rightarrow S[u_s]'\rightarrow S[v_r]'\rightarrow0$.
With the interpretation we adopted in Remark 3.3 and the notation we introduced in Remark 3.4, the preceding observation can be restated as follows: 
after making a standard change of basis around $a_{\tau\sigma}$, the resulting resolution admits the consecutive cancellation
$0\rightarrow S[\sigma]\rightarrow S[\tau]\rightarrow0$.
 \end{remark}
 
 We close this section introducing the following terminology. After making a standard change of basis around an invertible entry $a_{\tau\sigma}$ of a 
 resolution $\mathbb{F}$, we obtain a new resolution $\mathbb{F'}$ such that 
 $\mathbb{F} = \mathbb{F'} \oplus (0\rightarrow S[\sigma]\rightarrow S[\tau]\rightarrow0)$. From now on,
 the consecutive cancellation $0\rightarrow S[\sigma]\rightarrow S[\tau]\rightarrow0$ will be called \textbf{standard cancellation}, and we will say that 
 $\mathbb{F'}$ is obained from $\mathbb{F}$ by means of a standard cancellation.
 
\section{Dominant Ideals}
We are ready to address the study of our first family of monomial ideals, the dominant ideals. This study includes the construction of their 
minimal free resolutions as well as an analysis of their combinatorial properties.
\begin{definition}
 Given a set $G$ of monomials in $S$, we say that
 \begin{itemize}
  \item An element $m\in G$ has a \textbf{dominant variable} $x$ (with respect to $G$) if for all 
 $m'\in G\setminus\{m\}$, the exponent with which $x$ appears in the factorization of $m$ is larger than the exponent with which $x$ appears in the 
 factorization of $m'$; that is, there exists a positive $k$ such that
$x^k\mid m$ and $x^k\nmid m'$, for all $m'\neq m$.
\item An element $m\in G$ is a \textbf{dominant monomial} (with respect to $G$) if it has a dominant variable.
\item The set $G$ is a \textbf{dominant set} if every $m\in G$ is dominant.
\item A monomial ideal $M$ is a \textbf{dominant ideal} if its minimal generating set is dominant.
 \end{itemize}
\end{definition}
\begin{example}
The ideals $M_1=(x^3y, xy^2z,xz^2)$ and $M_2=(wx,y^3,z^2)$ are dominant, while $M_3=(x^2,y^2,xy)$ is not.
\end{example}

Some comments are in order. First, notice that the concept of dominant monomial always depends on a reference set. For example, the ideal $M_3$
introduced above is not dominant because $xy$ is not dominant in the minimal generating set $\{x^2,y^2,xy\}$; however, $xy$ is dominant in the proper 
subset $\{x^2,xy\}$.

Second, the definitions of dominant ideal and generic ideal are based on properties of the exponents of the monomial generators. 
(Recall that an ideal is generic if no variable appears with the same nonzero exponent in more than one monomial generator.) Despite this similarity, 
dominant and generic ideals are generally different. In Example 4.2, for instance, $M_1$ is dominant but not generic, while
$M_3$ is generic but not dominant. 

Finally, observe that if a monomial ideal is a complete intersection, its monomial generators are dominant because
they do not have variables in common (such is the case with $M_2$). It follows
that the ideal itself is dominant. Thus, monomial complete intersections are a subset of the family of dominant ideals.

Let us now study some properties derived from the concept of dominance. The following lemma will be quoted often throughout this work. 
\begin{lemma}
Let $M$ be a monomial ideal with minimal generating set $G$. If $[\sigma_1]$ and $[\sigma_2]$ are two basis elements of $\mathbb{T}_M$ with 
$\mdeg[\sigma_1]=\mdeg[\sigma_2]$, then $[\sigma_1]$ and $[\sigma_2]$ contain the same dominant monomials of $G$.
\end{lemma}
\begin{proof}
 Let $L_1$ and $L_2$ be the sets of monomials contained in $[\sigma_1]$ and $[\sigma_2]$, respectively. Then $\lcm(L_1)=\lcm(L_2)$.
 If neither $L_1$ nor $L_2$ contains dominant elements of $G$, there is nothing to prove. 
 
 Suppose now that one of these sets, call it $L_i$, contains a 
 dominant monomial $m$ of $G$. We will show that the other set, call it $L_j$, contains $m$ as well. Since $m$ has a dominant variable $x$, there is a
 positive $k$ such that $x^k\mid m$ and $x^k\nmid m'$, for all $m'$ in $G\setminus\{m\}$. In particular, $x^k\nmid m'$ for all $m'$ in $L_j\setminus\{m\}$.
 That is, $x^k\nmid \lcm(L_j\setminus\{m\})$. On the other hand, $x^k\mid \lcm(L_i)=\lcm(L_j)$.
 
 Hence, $L_j\neq L_j\setminus\{m\}$, which means that $m$ is in $L_j$. We have proven that each dominant element
 $m$ of $G$ which is in one of $[\sigma_1]$ and $[\sigma_2]$ is also contained in the other.
\end{proof} 
In the following theorem we construct the minimal resolutions of dominant ideals. This theorem yields, in addition, an explcit characterization of when 
the Taylor resolution is minimal.
\begin{theorem}
 Let $M$ be a monomial ideal. Then $\mathbb{T}_M$ is minimal if and only if $M$ is dominant.
 \end{theorem}
 \begin{proof}
  $(\Rightarrow)$ Suppose that $M$ is not dominant. Then its minimal generating set $G$ contains a nondominant monomial $n$. Let $\sigma=G$ and 
  $\tau_m=G\setminus\{m\}$. This means that $n\mid\
  lcm\left(\tau_n\right)$ and thus, $\mdeg[\sigma]=\mdeg\left[\tau_n\right]$. So, the top differential map sends $[\sigma]\mapsto\sum\limits_
  {m\neq n} a_m\left[\tau_m\right]\pm 1 \left[\tau_n\right]$.
  Since the coefficient $\pm1$ of $\left[\tau_n\right]$ is invertible, $\mathbb{T}_M$ is not minimal, a contradiction.\\
  $(\Leftarrow)$ If $[\sigma]=[m_1,\ldots, m_j]$ and $[\tau_i]=
  [m_1,\ldots,\widehat{m_i},\ldots,m_j]$ for all $i$, then 
  \[f_j\left([\sigma]\right)=\sum_{i=1}^ja_{\tau_i\sigma}[\tau_i],\]
  where $a_{\tau_i\sigma}=(-1)^{i+1}\dfrac{\mdeg[\sigma]}{\mdeg[\tau_i]}$. Since $m_i$ is dominant, 
  it follows from Lemma 4.3 that 
  $a_{\tau_i\sigma}$ is not invertible. This means that the differential matrices of $\mathbb{T}_M$ do not have invertible entries and hence, $\mathbb{T}_M$ 
  is minimal.
   \end{proof}
   \begin{corollary}
    Dominant ideals are Scarf.
   \end{corollary}
   \begin{proof}
   If two basis elements $[\sigma_1],[\sigma_2]$ of $\mathbb{T}_M$ have the same multidegree, according to Lemma 4.3, they contain the same
   dominant monomials. Since all monomials of the minimal generating set are dominant, $[\sigma_1]=[\sigma_2]$.  
    \end{proof}
It follows from Lemma 4.3 that if $M$ is dominant, no facet $[\tau_i]$ of $[\sigma]$ has the same multidegree as $[\sigma]$. However, Corollary 4.5 shows 
that an even stronger statement is true: if $M$ is dominant, all basis elements of $\mathbb{T}_M$ have different multidegrees.
\begin{example}
Let $M=(x^2,xy,y^3)$. The Taylor resolution $\mathbb{T}_M$ of $S/M$ is
\small\[0\rightarrow S[x^2,xy,y^3]\xrightarrow{\left(
\begin{array}{c}
x \\
-1\\
y^2
\end{array}\right)}
\begin{array}{c}
S[xy,y^3]\\
\oplus\\
S[x^2,y^3]\\
\oplus\\
S[x^2,xy]
\end{array}
\xrightarrow{\left(
\begin{array}{ccc}
0&-y^3&-y\\
-y^2&0&x\\
x&x^2&0
\end{array}\right)}
 \begin{array}{c}
S[x^2]\\
\oplus\\
S[xy]\\
\oplus\\
S[y^3]
\end{array}
\xrightarrow{\left(
\begin{array}{ccc}
x^2&xy&y^3
\end{array}\right)}
S[\varnothing]\rightarrow S/M\rightarrow 0\]
\end{example}
Notice that $M$ is not a dominant ideal since $xy$ is nondominant. It follows from Theorem 4.4 that $\mathbb{T}_M$ is not minimal,
which is consistent with the fact that one of the differential matrices contains an invertible entry $-1$.\\
In contrast to the previous example, the next one contains a Taylor Resolution which is minimal.
\begin{example} 
Let $M=(x^2,xz,y^3)$. the Taylor resolution $\mathbb{T}_M$ of $S/M$ is
\small\[0\rightarrow S[x^2,xz,y^3]\xrightarrow{\left(
\begin{array}{c}
x \\
-z\\
y^3
\end{array}\right)}
\begin{array}{c}
S[xz,y^3]\\
\oplus\\
S[x^2,y^3]\\
\oplus\\
S[x^2,xz]
\end{array}
\xrightarrow{\left(
\begin{array}{ccc}
0&-y^3&-z\\
-y^3&0&x\\
xz&x^2&0
\end{array}\right)}
 \begin{array}{c}
S[x^2]\\
\oplus\\
S[xz]\\
\oplus\\
S[y^3]
\end{array}
\xrightarrow{\left(
\begin{array}{ccc}
x^2&xz&y^3
\end{array}\right)}
S[\varnothing]\rightarrow S/M\rightarrow 0\]
\end{example}
In this example, $M$ is dominant. According to Theorem 4.4, the Taylor Resolution $\mathbb{T}_M$ is minimal, which is consistent with the fact that none of the 
differential matrices contains invertible entries.\\

Having obtained the minimal free resolutions of the dominant ideals, we can now study the combinatorial properties of the family. We will adopt the 
following notation: $\reg\left(S/M\right)$, $\pd(S/M)$, and 
$\bi(S/M)$ will represent the regularity, projective dimension, and $i^{th}$ Betti number of $S/M$,  respectively.

\begin{theorem}(Regularity of Dominant Ideals)\\
 Let $M$ be a dominant ideal with minimal generating set $G=\{m_1,\ldots,m_q\}$.\\
 Let $h=\deg\left(\mdeg[m_1,\ldots, m_q]\right)$. Then 
 $\reg\left(S/M\right)=h-q$.
\end{theorem}
\begin{proof}
 Since $[m_1,\ldots, m_q]$ is a basis element in homological degree $q$, it follows that $b_{qh}\neq 0$. Thus, $\reg\left(S/M\right)\geq h-q$. 
 We will prove that if $b_{ij}\neq 0$, then $h-q\geq j-i$, which will complete the proof.
 
 Let $[\sigma]=[m_{r_1},\ldots, m_{r_i}]$ be a basis element of $\mathbb{T}_M$ with $\deg\left(\mdeg[\sigma]\right)=j$. 
 Let $m\in G\diagdown\{m_{r_1},\ldots,m_{r_i}
 \}$. Since different monomial generators have different dominant variables, it follows that
 \[\deg\left(\mdeg[m_{r_1},\ldots, m_{r_i},m]\right)\geq \deg\left(\mdeg[m_{r_1},\ldots, m_{r_i}]\right)+1\]
 Then, after applying the preceding reasoning $q-i$ times, we get
 \begin{dmath*}
 h=\deg\left(\mdeg[m_1,\ldots, m_q]\right)=\deg\left(\mdeg[m_{r_1},\ldots, m_{r_i},m_{s_1},\ldots, m_{s_{q-i}}]\right)\geq
 \deg\left(\mdeg[m_{r_1},\ldots, m_{r_i}]\right)+(q-i)=j+q-i
 \end{dmath*} 
 This implies that $h-q\geq j-i$.
\end{proof}
\begin{corollary}(Characterization of the minimal Taylor Resolution)\\ 
Let $M$ be a monomial ideal minimally generated by $q$ monomials. The following statements are equivalent:
\begin{enumerate}[(i)]
 
 \item $\mathbb{T}_M$ is minimal.
\item M is dominant.
\item $\bi(S/M)={q\choose i}$ for all $i$.
\item $\pd(S/M)=q$.
\item The LCM lattice of M is boolean.
\end{enumerate}
\end{corollary}
\begin{proof}
The equivalence of (i), (ii), (iii) and (v) is immediate, as is (iii) $\Rightarrow$ (iv).
We complete the proof by showing that (iv) $\Rightarrow$ (i).

Assume that the Taylor Resolution is not minimal. Then, by Theorem 4.4, $M$ is not dominant. Thus there exists a nondominant monomial $m$ in the
minimal generating set $G$ of $M$. Let $\sigma = G$ and $\tau = G\setminus\{m\}$. Then $m\mid \text{lcm}(\tau)$ and hence, 
$\mdeg[\sigma]=\mdeg[\tau]$. Since $[\sigma]$ and $[\tau]$ are face and facet in homological degrees $q$ and $q-1$ respectively, it follows
that the $q^{th}$ differential matrix $(d_q)$ of $\mathbb{T}_M$ contains an invertible entry.
After making a consecutive cancellation in homological degrees $q$ and $q-1$, we obtain a new resolution $\mathbb{F}$ of $S/M$.
But the rank of the free module in homological degree $q$ of $\mathbb{T}_M$ is 1, which implies that the rank of the free module in homological degree $q$ 
of $\mathbb{F}$ is 0. Hence, the length of $\mathbb{F}$ is less than $q$, a contradiction.
\end{proof}

The following two remarks are now trivial but show that dominant ideals are as good as we could expect. First, note that the Taylor 
resolution of $S/M$ agrees with the Scarf complex of $S/M$ if and only if $M$ is dominant. This is interesting because the Taylor resolution
is usually highly nonminimal, while the Scarf complex is often strictly contained in the minimal free resolution of $S/M$. Second, two dominant
ideals whose minimal generating sets have the same cardinality must have the same projective dimension and the same total Betti numbers. This is immediate
from Corollary 4.9 (iii) and (iv).
\section{Semidominant Ideals}
In this section we introduce the 
semidominant ideals by slightly modifying the definition of dominance in such a way that the resulting family does not overlap with the family of dominant
ideals and yet retains some of its rich properties.
\begin{definition}
Let $G$ be a set of monomials in $S$. We say that $G$ is \textbf{semidominant} if exactly one monomial of $G$ is not dominant. A monomial ideal $M$ is 
called a \textbf{semidominant ideal} if its minimal generating set is semidominant. When a semidominant set $G$ is expressed in the form 
$G=\{m_1,\ldots,m_q,n\}$ we will assume that $m_1,\ldots,m_q$ are dominant and $n$ is nondominant.
\end{definition}
\begin{example}
The ideals $M_1=(x^2,y^3, xy)$ and $M_2=(xy,z^2,yz)$ are semidominant, $M_3=(x^2z,y^3, yz^3)$ is dominant, and $M_4=(xy,yz,xz)$ 
is neither dominant nor semidominant.\\
\end{example}

Note that the concept of semidominance is obtained from that of dominance in the same way as the definition of almost complete intersection 
is derived from that of complete intersection; namely, by relaxing the defining conditions. In the next proposition we explain how the former concepts 
extend the latter.

\begin{proposition}
Monomial almost complete intersections are either dominant or semidominant ideals.
\end{proposition}
\begin{proof}
Let $M=(l_1,...,l_q,l)$ be a monomial almost complete intersection, where $l_1, ...,l_q$ form a regular sequence and hence have no variable in common. 
Note that for all $i$, $l_i\nmid l$. Then there is a variable $x_i$ that appears with a larger exponent in the factorization of $l_i$ than in that of $l$. 
Therefore, $x_i$ is a dominant variable for $l_i$, which means that $l_i$ is a dominant monomial.
\end{proof}

Observe that the two cases stated in the proposition are feasible (consider $M_2$ and $M_3$ in Example 5.2). Later, we will prove that semidominant ideals 
are Scarf which, combined with Corollary 4.5 and Proposition 5.3, implies that monomial almost complete intersections are Scarf too.

Now we are ready to construct the minimal free resolutions of semidominant ideals. The idea is simple: if $M$ is semidominant and we identify the basis of
$\mathbb{T}_M$ with the full simplex on $M$, we will prove that the basis of the minimal free resolution of $S/M$ can be obtained by eliminating pairs
$\left([\sigma],[\tau]\right)$ of face and facet of equal multidegree from the simplicial complex in arbitrary order until we exhaust all such pairs. We
begin with a lemma.
\begin{lemma}
 Let $M$ be a semidominant ideal. Let $\mathbb{F}$ be a free resolution of $S/M$ obtained from $\mathbb{T}_M$ by means of standard cancellations.
 If two basis elements of $\mathbb{F}$ have the same multidegree, then they are face and facet.
\end{lemma}
\begin{proof}
 Let $[\sigma]$ and $[\tau]$ be two basis elements of $\mathbb{F}$. If $\mdeg[\sigma]=\mdeg[\tau]$ then, according to Lemma 4.3, $[\sigma]$ and 
 $[\tau]$ contain the same dominant monomials, and thus they must differ in the nondominant monomials that define them. Since the minimal generating set of 
 $M$ contains exactly one nondominant monomial $n$, we conclude that one of these basis elements contains $n$ while the other does not. 
 That is, $[\sigma]$ and $[\tau]$ are face and facet.
\end{proof}
The next two results show that, in the context of semidominant ideals, the process of eliminating pairs of face and facet of equal multidegree is equivalent 
to that of making standard cancellations.

Note: We will say that two pairs of basis elements  $\left([\sigma],[\tau]\right)$ and $\left([\theta],[\pi]\right)$ of $\mathbb{T}_M$ 
are ``disjoint'' if $[\sigma]\neq[\theta],[\pi]$ and $[\tau]\neq[\theta],[\pi]$.

\begin{lemma}
 Let $M$ be a semidominant ideal. Let $\mathbb{F}$ be a free resolution of $S/M$ obtained from $\mathbb{T}_M$ by means of standard cancellations. Let
 $a_{\tau\sigma}$ and $a_{\pi\theta}$ be two invertible entries of $\mathbb{F}$, determined by two disjoint pairs of basis elements 
 $\left([\sigma],[\tau]\right)$ and $\left([\theta],[\pi]\right)$ of $\mathbb{F}$, respectively.\\
 Then after making the standard cancellation $0\rightarrow S[\sigma]\rightarrow S[\tau]\rightarrow0$ in $\mathbb{F}$, it is possible to make the 
 standard cancellation $0\rightarrow S[\theta]\rightarrow S[\pi]\rightarrow0$.
\end{lemma}
\begin{proof}
 $[\sigma]$ and $[\tau]$ are basis elements in homological degrees $j$ and $j-1$, respectively, for some $j$. Thus $a_{\tau\sigma}$ is an entry of the 
 differential matrix $\left(f_j\right)$ of $\mathbb{F}$. Similarly, $[\theta]$ and $[\pi]$ are basis elements in some homological degrees $k$ and $k-1$,
 and $a_{\pi\theta}$ is an entry of the differential matrix $\left(f_k\right)$ of $\mathbb{F}$.
 
 In order to prove the lemma, it is enough to show that after making the standard cancellation 
 $0\rightarrow S[\sigma]\rightarrow S[\tau]\rightarrow0$ in $\mathbb{F}$, the entry $a'_{\pi\theta}$ of the differential matrix $\left(f'_k\right)$ of the 
 new resolution $\mathbb{F}'$ is invertible.
 
 Given that only $\left(f_{j+1}\right)$, $\left(f_j\right)$ and $\left(f_{j-1}\right)$ are affected by the standard cancellation
 $0\rightarrow S[\sigma]\rightarrow S[\tau]\rightarrow0$, if $k\neq j-1,j,j+1$ then $a'_{\pi\theta}=a_{\pi\theta}$; that is, $a'_{\pi\theta}$ is invertible.
 Therefore, we only need to prove that $a'_{\pi\theta}$ is invertible in the following cases:\\
 $k=j$; $k=j-1$, and $k=j+1$.
 
 First, suppose $k=j$. Since $a_{\pi\theta}$ is invertible, $\mdeg[\pi]=\mdeg[\theta]$. Then  $a'_{\pi\theta}=0$ or  $a'_{\pi\theta}$ is invertible. 
 Let us assume that $a'_{\pi\theta}=0$.
 By Lemma 3.2 (iii), we have that $0=a'_{\pi\theta}=a_{\pi\theta}-\dfrac{a_{\pi\sigma}a_{\tau\theta}}{a_{\tau\sigma}}$. It follows that  
 $a_{\pi\theta}a_{\tau\sigma}= a_{\pi\sigma}a_{\tau\theta}$ and, since  $a_{\pi\theta}$ and $a_{\tau\sigma}$ are invertible, 
  $a_{\pi\sigma}$ and $a_{\tau\theta}$ must be invertible too. In particular, the fact that $a_{\pi\sigma}$ is invertible 
  implies that $\mdeg[\sigma]=\mdeg[\pi]$ which, combined with the hypothesis $\mdeg[\sigma]=\mdeg[\tau]$, implies that $\mdeg[\tau]=\mdeg[\pi]$.
 It follows from Lemma 5.4 that one of $[\tau]$ and $[\pi]$ is a face and the other is its facet. Then they must appear in consecutive homological degrees,
  which is absurd because $k=j$. We conclude that $a'_{\pi\theta}$ is invertible.
  
  Now suppose $k=j-1$. In this case $[\tau]$ and $[\theta]$ appear in homological degree $j-1$. Let $[\tau]$ and $[\theta]$ be the $r^{th}$ and $s^{th}$ basis elements,
  respectively. It follows from Lemma 3.2 iv) that after making the standard cancellation $0\rightarrow S[\sigma]\rightarrow S[\tau]\rightarrow0$, the 
  matrix $\left(f'_{j-1}\right)$ of the new resolution $\mathbb{F}'$ is obtained from $\left(f_{j-1}\right)$ by eliminating its $r^{th}$ column. Since 
  the entry $a'_{\pi\theta}$ is placed in the $s^{th}$ column of $\left(f'_{j-1}\right)$, we have that $a'_{\pi\theta}=a_{\pi\theta}$; that is, 
  $a'_{\pi\theta}$ is invertible.
  
  Finally, suppose $k=j+1$. In this case $[\sigma]$ and $[\pi]$ appear in homological degree $j$. Let $[\sigma]$ and $[\pi]$ be the $u^{th}$ and $v^{th}$ basis elements,
  respectively. It follows from Lemma 3.2 iv) that after making the standard cancellation $0\rightarrow S[\sigma]\rightarrow S[\tau]\rightarrow0$, the 
  matrix $\left(f'_{j+1}\right)$ of the new resolution $\mathbb{F}'$ is obtained from $\left(f_{j+1}\right)$ by eliminating its $u^{th}$ row. Since 
  the entry $a'_{\pi\theta}$ is placed in the $v^{th}$ row of $\left(f'_{j+1}\right)$, we have that $a'_{\pi\theta}=a_{\pi\theta}$; that is, 
  $a'_{\pi\theta}$ is invertible.
 \end{proof}
\begin{theorem}
 Let $M$ be a semidominant ideal. Let $\left([\sigma_1],[\tau_1]\right),\ldots,\left([\sigma_k],[\tau_k]\right)$ be $k$ pairs of basis elements of 
 $\mathbb{T}_M$, satisfying the following properties:
 \begin{enumerate}[(i)]
  \item $\left([\sigma_i],[\tau_i]\right)$ and $\left([\sigma_j],[\tau_j]\right)$ are disjoint, if $i\neq j$.
 \item $[\tau_i]$ is a facet of $[\sigma_i]$, for all $i=1,\ldots, k$.
 \item $\mdeg[\sigma_i]=\mdeg[\tau_i]$, for all $i=1,\ldots, k$.
 \end{enumerate}
 
 Then, starting with $\mathbb{T}_M$ it is possible to make the following sequence of standard cancellations:
 \[0\rightarrow S[\sigma_1]\rightarrow S[\tau_1]\rightarrow0,\quad\cdots\quad,0\rightarrow S[\sigma_k]\rightarrow S[\tau_k]\rightarrow0\]
\end{theorem} 
\begin{proof}
 The proof is by induction on $k$.\\
 If $k=2$, the statement holds by Lemma 5.5, with $\mathbb{F}=\mathbb{T}_M$. (The fact that $a_{\tau_1\sigma_1}$ and $a_{\tau_2\sigma_2}$ are invertible 
 follows from the fact that in $\mathbb{T}_M$ faces and facets of equal multidegree always determine an invertible entry.)
 
 Assume that the theorem holds for $k=j-1$. Let $k=j$. Then it is possible to make either of the following two sequences of standard cancellations:
 \[0\rightarrow S[\sigma_1]\rightarrow S[\tau_1]\rightarrow0,\quad\cdots\quad,0\rightarrow S[\sigma_{j-1}]\rightarrow S[\tau_{j-1}]\rightarrow0\]
 and
 \[0\rightarrow S[\sigma_1]\rightarrow S[\tau_1]\rightarrow0,\quad\cdots\quad,0\rightarrow S[\sigma_{j-2}]\rightarrow S[\tau_{j-2}]\rightarrow0,
 0\rightarrow S[\sigma_j]\rightarrow S[\tau_j]\rightarrow0.\]
 This means that after making the first $j-2$ cancellations
 \[0\rightarrow S[\sigma_1]\rightarrow S[\tau_1]\rightarrow0,\quad\cdots\quad,0\rightarrow S[\sigma_{j-2}]\rightarrow S[\tau_{j-2}]\rightarrow0\]
either of the following two cancellations can be made: \[0\rightarrow S[\sigma_{j-1}]\rightarrow S[\tau_{j-1}]\rightarrow0\] and \[0\rightarrow S[\sigma_j]
\rightarrow S[\tau_j]\rightarrow0.\]

In other words, after making the first $j-2$ standard cancellations, we obtain a free resolution $\mathbb{F}$, where the entries 
$a_{\tau_{j-1}\sigma_{j-1}}$ and $a_{\sigma_j\tau_j}$ determined by $\left([\sigma_{j-1}],[\tau_{j-1}]\right)$ and $\left([\sigma_j],[\tau_j]\right)$,
respectively, are invertible. Therefore, it follows from Lemma 5.5, that after making the cancellation $0\rightarrow S[\sigma_{j-1}]\rightarrow 
S[\tau_{j-1}]\rightarrow0$, the cancellation $0\rightarrow S[\sigma_j]\rightarrow S[\tau_j]\rightarrow0$ is still possible.
\end{proof}
 Note. In Proposition 5.6, the pairs $\left([\sigma_1],[\tau_1]\right),\ldots,\left([\sigma_k],[\tau_k]\right)$ are indistinguishable,
 which implies that the standard cancellations can be made in arbitrary order.

\begin{lemma}
 Let $M=(m_1,\ldots,m_q,n)$ be a semidominant ideal. Let \\
 $A=\left\{\left([m_{i_1},\ldots, m_{i_j},n],[m_{i_1},\ldots, m_{i_j}]\right):n\mid\lcm(m_{i_1},\ldots,m_{i_j})\right\}$. 
 Then the following properties are satisfied:
 \begin{enumerate}[(i)]
\item If $\left([\sigma_1],[\tau_1]\right)$ and $\left([\sigma_2],[\tau_2]\right)$ are distinct ordered pairs of $A$, then they are disjoint.
\item $[\tau]$ is a facet of $[\sigma]$, for all $\left([\sigma],[\tau]\right)\in A$.
\item $\mdeg[\sigma]=\mdeg[\tau]$, for all $\left([\sigma],[\tau]\right)\in A$.
\item If $\left([\sigma],[\tau]\right)$ is an ordered pair of basis elements of $\mathbb{T}_M$ such that $[\tau]$ is a facet of $[\sigma]$ and 
 $\mdeg[\sigma]=\mdeg[\tau]$, then $\left([\sigma],[\tau]\right)\in A $.
 \end{enumerate}
 \end{lemma}
\begin{proof}
 (i) Since $[\sigma_1]$ and $[\sigma_2]$ contain $n$ and $[\tau_1]$ and $[\tau_2]$ do not contain $n$, it follows that $[\sigma_1]\neq[\tau_2]$ and
 $[\tau_1]\neq[\sigma_2]$. Let us assume that $[\sigma_1]=[\sigma_2]$. Then, by construction, $[\tau_1]=[\tau_2]$ and thus
 $\left([\sigma_1],[\tau_1]\right)=\left([\sigma_2],[\tau_2]\right)$, an absurd. Let us now assume that $[\tau_1]=[\tau_2]$. Then, by construction,
 $[\sigma_1]=[\sigma_2]$ and thus  $\left([\sigma_1],[\tau_1]\right)=\left([\sigma_2],[\tau_2]\right)$, a contradiction.\\
 (ii) Trivial.\\
 (iii) Since $n\mid\lcm\left(m_{i_1},\ldots,m_{i_j}\right)$, it follows that $\lcm\left(m_{i_1},\ldots,m_{i_j}\right)=
 \lcm\left(m_{i_1},\ldots,m_{i_j},n\right)$.\\
 (iv) If $\mdeg[\sigma]=\mdeg[\tau]$ then, by Lemma 4.3, $[\sigma]$ and $[\tau]$ contain the same dominant monomials, and therefore they differ in the 
 nondominant monomials that define them. But the minimal generating set of $M$ contains exactly one nondominant monomial and $[\tau]$ is a facet of $[\sigma]$, which implies
 that $[\sigma]$ and $[\tau]$ must be of the form $[\sigma]=[m_{i_1},\ldots, m_{i_j},n]$; $[\tau]=[m_{i_1},\ldots, m_{i_j}]$.
\end{proof}
\begin{theorem}
 Let $M=(m_1,\ldots,m_q,n)$ be a semidominant ideal. Let\\
 $A=\left\{\left([m_{i_1},\ldots, m_{i_j},n],[m_{i_1},\ldots, m_{i_j}]\right):n\mid\lcm(m_{i_1},\ldots,m_{i_j})\right\}$. Then 
 the minimal free resolution of $S/M$ can be obtained from $\mathbb{T}_M$ by doing all standard cancellations $0\rightarrow S[\sigma]\rightarrow S[\tau]
 \rightarrow0$, with $\left([\sigma],[\tau]\right)\in A$. In other words, if $\mathbb{F}$ is the minimal free resolution of $S/M$, then 
 \[\mathbb{T}_M=\mathbb{F}\oplus\left(\bigoplus_{\left([\sigma],[\tau]\right)\in A}0\rightarrow S[\sigma]\rightarrow S[\tau]\rightarrow0\right)\]
\end{theorem}
\begin{proof}
 Notice that the ordered pairs of $A$ satisfy the hypotheses of Theorem 5.6, by Lemma 5.7. Therefore, starting with
  $\mathbb{T}_M$, it is possible to make all standard cancellations $0\rightarrow S[\sigma]\rightarrow S[\tau]\rightarrow0$, with 
 $\left([\sigma],[\tau]\right)\in A$. We claim that the free resolution $\mathbb{F}$, obtained after making all these cancellations, is minimal.
 
 Let us assume that $\mathbb{F}$ is not minimal. Then there exists an invertible entry $a_{\tau\sigma}$ of $\mathbb{F}$, determined by two basis elements
 $[\sigma]$ and $[\tau]$ of $\mathbb{F}$. Hence, $[\sigma]$ and $[\tau]$ have the same multidegree. Thus by Lemma 5.4, $[\sigma]$ and $[\tau]$ are face and 
 facet. It follows from Lemma 5.7 (iv) that $\left([\sigma],[\tau]\right)\in A$, a contradiction. 
\end{proof}
\begin{corollary}
 Semidominant ideals are Scarf.
\end{corollary}
\begin{proof}
 Let $M$ be a semidominant ideal. If $[\sigma]$ and $[\tau]$are basis elements of $\mathbb{T}_M$ and $\mdeg[\sigma]=\mdeg[\tau]$, then by 
 Lemma 5.4 we have that $[\sigma]$ and $[\tau]$ are face and facet. It follows from Lemma 5.7 (iv) and Theorem 5.8 that $[\sigma]$ and $[\tau]$ are excluded 
 from the minimal free resolution of $S/M$.
\end{proof}

Since the Scarf complex of an ideal is the intersection of all its minimal resolutions (as proved in  [Me]), it follows that all minimal resolutions of 
semidominant ideals have the same basis.

\begin{example}
 Let $M=(x^3y,y^2z,xz^2,xyz)$. Note that $M$ is semidominant, $xyz$ being the nondominant generator. By Corollary 5.9, $M$ is Scarf. Now, the multidegrees
 that are common to more than one basis element of $\mathbb{T}_M$ are $x^3y^2z$, $x^3yz^2$, $xy^2z^2$, and $x^3y^2z^2$ as one can determine by simple
 inspection. Hence, the basis of the minimal resolution $\mathbb{F}$ of $S/M$ is obtained from the basis of $\mathbb{T}_M$ by eliminating the elements that
 have one of the multidegrees mentioned above. This leads to the following resolution:
 \[\mathbb{F}:\quad 0\rightarrow 
\begin{array}{c}
S[x^3y,xyz] \\
\oplus\\
S[y^2z,xyz]\\
\oplus\\
S[xz^2,xyz]
\end{array}\xrightarrow{\left(f_2\right)}
\begin{array}{c}
S[x^3y]\\
\oplus\\
S[y^2z]\\
\oplus\\
S[xz^2]\\
\oplus\\
S[xyz]
\end{array}
\xrightarrow{\left(f_1\right)}
 S[\varnothing]\xrightarrow{\left(f_0\right)} S/M\rightarrow 0\] 
\end{example}

\begin{corollary}
 Let $M$ be a semidominant ideal with minimal generating set $G=\{m_1,\ldots,m_q,n\}$.
 \begin{enumerate}[(i)]
 \item The projective dimension of $S/M$ is the cardinality of the largest dominant subset of $G$ that contains $n$.
 \item Let $B_j=\left\{[m_{t_1},\ldots, m_{t_j}]: n\nmid\mdeg[m_{t_1},\ldots, m_{t_j}]\right\}$. Then the total Betti numbers are given by the formula
  \[\bi\left(S/M\right)=\#B_i+\#B_{i-1}\]
 \end{enumerate}
  \end{corollary}
\begin{proof}
 Let $\mathbb{F}$ and $A$ be as in Theorem 5.8.\\
 (i) Let $r=\max\left\{\#(D): D\text{ is a dominant subset of }G\text{ that contains }n\right\}$.\\ 
 Let $\left\{m_{t_1},\ldots,m_{t_{r-1}},n\right\}$
 be a dominant subset of $G$. Then $n\nmid\lcm\left(m_{t_1},\ldots,m_{t_{r-1}}\right)$. Thus $\left([m_{t_1},\ldots, m_{t_{r-1}},n],
 [m_{t_1},\ldots, m_{t_{r-1}}]\right)$ is not in $A$ and, therefore, $[m_{t_1},\ldots, m_{t_{r-1}},n]$ is a basis element of the minimal resolution
  $\mathbb{F}$. Thus, $\pd\left(S/M\right)\geq r$. Now, if $[\sigma]$ is a basis element of $\mathbb{T}$, in homological degree $k>r$, then $[\sigma]$
  must be of the form: $[\sigma]=[m_{s_1},\ldots, m_{s_k}]$ or $[\sigma]=[m_{s_1},\ldots, m_{s_{k-1}},n]$.
  
  If $[\sigma]=[m_{s_1},\ldots, m_{s_k}]$, then $\{m_{s_1},\ldots, m_{s_k},n\}$ cannot be dominant because its cardinality is larger than $r$. Hence, 
  $n\mid \lcm(m_{s_1},\ldots,m_{s_k})$, which means that $\left([m_{s_1},\ldots, m_{s_k},n],[\sigma]\right)\in A$, and thus $[\sigma]$ is not a basis
  element of $\mathbb{F} $. A similar reasoning shows that if $[\sigma]=[m_{s_1},\ldots, m_{s_{k-1}},n]$ then
  $\left([\sigma],[m_{s_1},\ldots, m_{s_{k-1}}]\right)\in A$, and thus 
  $[\sigma]$ is not a basis element of $\mathbb{F}$.
  
  Given that every basis element of $\mathbb{T}_M$ in homological degree $k>r$ is excluded from the basis of $\mathbb{F}$, we conclude that 
  $\pd\left(S/M\right)=r$.\\
  (ii) The basis elements of $\mathbb{T}_M$ in homological degree $i$ are of the form $[m_{s_1},\ldots, m_{s_{i-1}},n]$ or $[m_{t_1},\ldots, m_{t_i}]$. 
  Since the basis elements of $\mathbb{F}$ are obtained from the basis of $\mathbb{T}_M$ 
  by eliminating those elements which are the first or the second component of a pair $\left([\sigma],[\tau]\right)\in A$, it follows that the family 
  of basis elements of $\mathbb{F}$ in homological degree $i$ is:\\
  $\left\{[m_{t_1},\ldots, m_{t_i}]: n\nmid\lcm\left(m_{t_1},\ldots,m_{t_i}\right)\right\}\cup
   \left\{[m_{s_1},\ldots, m_{s_{i-1}},n]: n\nmid\lcm\left(m_{s_1},\ldots,m_{s_{i-1}}\right)\right\}$.\\
   The statement of part (ii) is now clear.
\end{proof}
\begin{corollary}
 Let $M=(m_1,\ldots,m_q,n)$ be a semidominant ideal. Then $\pd\left(S/M\right)=2$ if and only if for all $i\neq j$, $n\mid \lcm(m_i,m_j)$.
\end{corollary}
\begin{proof}
 $(\Rightarrow)$ If $\pd\left(S/M\right)=2$, then the largest dominant subset of $\{m_1,\ldots,m_q,n\}$ that contains $n$ has cardinality 2 
 (Corollary 5.11). Thus every set $\{m_i,m_j,n\}$ is nondominant, which implies that $n\mid \lcm(m_i,m_j)$.\\
 $(\Leftarrow)$ If $k\geq2$, then $n\mid \lcm(m_{i_1},\ldots,m_{i_k})$. Therefore, the set $D=\{m_{i_1},\ldots,m_{i_k},n\}$ is not dominant and, according
 to Corollary 5.11, $\pd\left(S/M\right)\leq 2$. Now, $\{m_1,n\}$ is dominant, so $\pd\left(S/M\right)=2$. 
\end{proof}

Corollary 5.12 is interesting because it tells us that an ideal $M$ may have maximum projective dimension (i.e., $\pd\left(S/M\right)=$ number of generators of 
$M$) and another ideal $M'$, obtained by adding one generator to the minimal generating set of $M$, may have minimum projective dimension (i.e., 
$\pd\left(S/M'\right)=2$). The next example illustrates this phenomenon.
\begin{example}
 Let $M=(v^2xyz,vw^2yz,vwx^2z,vwxy^2,wxyz^2)$, and $M'=(v^2xyz,vw^2yz,$\\$vwx^2z,vwxy^2,wxyz^2,vwxyz)$. Since $M$ is dominant, $\pd\left(S/M\right)=5$. 
 The semidominant ideal $M'$ obtained from $M$ by adding the generator $vwxyz$ satisfies the condition of Corollary 5.12
 and thus $\pd\left(S/M'\right)=2$.
\end{example}
\begin{corollary} 
Let $M$ be a semidominant ideal with minimal generating set 
$G=\{m_1,\ldots,$\\$m_q,n\}$. Then \[\reg\left(S/M\right)=\max\left\{\deg\left(\mdeg[\sigma]\right)-
\hdeg[\sigma]:\sigma\subset G\text{, }n\in\sigma\text{, and }\sigma\text{ is dominant}\right\}\]   
\end{corollary}
\begin{proof}
 Let $\{m_{r_1},\ldots,m_{r_t},n\}$ be a dominant set such that 
 \[\deg\left(\mdeg[m_{r_1},\ldots, m_{r_t},n]\right)-(t+1)=c.\] 
 Then $\reg\left(S/M\right)\geq c$. We will prove that if $b_{ij}\neq 0$, then $c\geq j-i$, which will complete the proof.
 
 There are two ways in which we might have $b_{ij}\neq 0$:\\ 
 (i) the minimal free resolution contains a basis element of the form $[m_{r_1},\ldots, m_{r_i}]$ such that $\{m_{r_1},\ldots,m_{r_i},n\}$ is dominant and
 $\deg\left(\mdeg[m_{r_1},\ldots, m_{r_i}]\right)=j$;\\
 (ii) the minimal free resolution contains a basis element of the form $[m_{s_1},\ldots, m_{s_{i-1}},n]$ such that $\{m_{s_1},\ldots, m_{s_{i-1}},n\}$ is 
 dominant and $\deg\left(\mdeg[m_{s_1},\ldots, m_{s_{i-1}},n]\right)=j$.\\
 If (i) happens, then $[m_{r_1},\ldots, m_{r_i},n]$ is also in the minimal free resolution and \[\deg\left(\mdeg[m_{r_1},\ldots, m_{r_i},n]\right)\geq
 \deg\left(\mdeg[m_{r_1},\ldots, m_{r_i}]\right)+1.\]
 It follows from the construction of $c$ that 
 \[c\geq \deg\left(\mdeg[m_{r_1},\ldots, m_{r_i},n]\right)-(i+1)\geq \deg\left(\mdeg[m_{r_1},\ldots, m_{r_i}]\right)+1-(i+1)=j-i.\]
 If (ii) happens, then it follows from the construction of $c$ that \[c\geq \deg\left(\mdeg[m_{s_1},\ldots, m_{s_{i-1}},n]\right)-i=j-i.\]
\end{proof}
\begin{example}
 Let $M=(x^3y,y^2z,xz^2,xyz)$ as in Example 5.10. Since we already know the minimal free resolution $\mathbb{F}$ of $S/M$, we can read off the numbers
 $\pd\left(S/M\right)$, $\bi\left(S/M\right)$, and $\reg\left(S/M\right)$ from $\mathbb{F}$. However, we will calculate these numbers using Corollary 5.11
 and Corollary 5.14 which, in some cases, turns out to be a faster alternative.\\
 Observe that the largest dominant sets containing the nondominant generator
 $xyz$ are $\{x^3y,xyz\}$, $\{y^2z,xyz\}$, and $\{xz^2,xyz\}$. It follows from Corollary 5.11 (i) that $\pd\left(S/M\right)=2$.\\ 
 Besides that, according to Corollary 5.11 (ii), $\btwo\left(S/M\right)$ is given by the formula:
 \begin{dmath*}
 \btwo\left(S/M\right)=\#\{[m_i,m_j]/n\nmid\mdeg[m_i,m_j]\}+\#\{[m_i,n]/n\nmid\mdeg[m_i]\}=
 \#\{\}+\#\{[x^3y,xyz];[y^2z,xyz];[xz^2,xyz]\}=3.
 \end{dmath*} 
   ($\bone\left(S/M\right)$ and $\bzero\left(S/M\right)$ are always easily obtained from $\mathbb{T}_M$.)   
   Finally, by Corollary 5.14 we have
   \[\reg\left(S/M\right)=\max\{\deg(\mdeg[x^3y,xyz])-2;
    \deg(\mdeg[y^2z,xyz])-2;
     \deg(\mdeg[xz^2,xyz])-2\}\]
     $=\max\{5-2; \ 4-2; \ 4-2\}=3$.
     
     All our calculations are consistent with the information encoded in $\mathbb{F}$, as we can easily verify. 
\end{example}

\section{2-Semidominant Ideals}
 The concepts of dominance and semidominance lead in a natural way to the more general definition of $p$-semidominance, which we give next.
\begin{definition}
 A set of monomials is called \textbf{$p$-semidominant} if it contains exactly $p$ nondominant monomials. A monomial ideal is called 
 \textbf{$p$-semidominant}  if its minimal generating set is $p$-semidominant. 
 \end{definition} 
 
 With this definition, dominant and semidominant ideals can be thought of as being 0-semidominant and 1-semidominant, 
 respectively. Sometimes, the word semidominant is used to denote 1-semidominant ideals while other times it makes reference to p-semidominant ideals in 
 general (as in the title of this paper). The meaning will be clear from the context.  
  
  In this section we will construct the minimal free resolution of 2-semidominant ideals; that is, monomial ideals $M$ with minimal
  generating set $G=\{m_1,\ldots,m_q,n_1,n_2\}$ where $m_1,\ldots,m_q$ are dominant and $n_1$ and $n_2$ are nondominant.  
  First, we want to know the character of the entries of the differential matrices of $\mathbb{T}_M$.
    \begin{lemma}
   Let $M$ be a 2-semidominant ideal. If two basis elements of a resolution of $S/M$, in consecutive homological degrees, have the same multidegree, then
   they are face and facet.
  \end{lemma}
  \begin{proof}
   Let $[\sigma]$ and $[\tau]$ be basis elements in homological degrees $j+1$ and $j$, respectively. If $\mdeg[\sigma]=\mdeg[\tau]$, then 
   $[\sigma]$ and $[\tau]$ must be generated by the same dominant monomials.
   Given that $[\sigma]$ has one more generator than $[\tau]$, if $[\tau]$ contains no nondominant generator, $[\sigma]$ must contain exactly one.
   On the other hand, if $[\tau]$ contains one nondominant generator, then $[\sigma]$ must contain both nondominant generators. The possibilities are 
   four:
   \begin{enumerate}[(i)]
    \item $[\tau]=[m_{i_1},\ldots, m_{i_j}]$; $[\sigma]=[m_{i_1},\ldots, m_{i_j},n_1]$;\\
    \item $[\tau]=[m_{i_1},\ldots, m_{i_j}]$; $[\sigma]=[m_{i_1},\ldots, m_{i_j},n_2]$;\\
    \item $[\tau]=[m_{i_1},\ldots, m_{i_{j-1}},n_1]$; $[\sigma]=[m_{i_1},\ldots, m_{i_{j-1}},n_1,n_2]$;\\
    \item $[\tau]=[m_{i_1},\ldots, m_{i_{j-1}},n_2]$; $[\sigma]=[m_{i_1},\ldots, m_{i_{j-1}},n_1,n_2]$.\\   
   \end{enumerate}
   
    In every case we see that $[\tau]$ is a facet of $[\sigma]$.
  \end{proof}

Our next goal is to prove that the basis of the minimal free resolution of $S/M$ can be obtained from the basis of its Taylor resolution by eliminating
pairs of basis elements $[\sigma]$, $[\tau]$ in an arbitrary order, where $[\tau]$ is a facet of $[\sigma]$ and $\mdeg[\sigma]=\mdeg[\tau]$,
until exhausting all possibilities.

If this idea is going to succeed, we need first to confirm that the following dangerous scenario never occurs. Suppose that 
$\left([\sigma_1],[\tau_1]\right)$ and $\left([\sigma_2],[\tau_2]\right)$ are disjoint pairs of face and facet with $\mdeg[\sigma_i]=\mdeg[\tau_i]$. 
Let $\left([\sigma_1],[\tau_1]\right)$ determine the invertible entry $a_{rs}$ of the differential matrix $\left(f_{j+1}\right)$ of $\mathbb{T}_M$.
Then eliminating $[\sigma_1]$ and $[\tau_1]$ from the basis of $\mathbb{T}_M$ is equivalent to making the standard change of basis around $a_{rs}$,
followed by the standard cancellation $0\rightarrow S[\sigma_1]\rightarrow S[\tau_1]\rightarrow0$.

Similarly, $\left([\sigma_2],[\tau_2]\right)$ define
an invertible entry $a_{cd}$ and eliminating $[\sigma_2],[\tau_2]$ from the basis of the Taylor resolution is equivalent to making a standard change of
basis around $a_{cd}$, followed by the standard cancellation $0\rightarrow S[\sigma_2]\rightarrow S[\tau_2]\rightarrow0$.

However, when we make the standard change of basis around $a_{rs}$, the entries of the matrices change. In particular, the entry $a_{cd}$ might become
noninvertible, which would prevent us from doing the standard cancellation $0\rightarrow S[\sigma_2]\rightarrow S[\tau_2]\rightarrow0$.

In the next lemma, which is analogous to Lemma 5.5, we show that this scenario is not possible for 2-semidominant ideals.
\begin{lemma}
 Let $M$ be a 2-semidominant ideal. Let $\mathbb{F}$ be a free resolution of $S/M$ obtained from $\mathbb{T}_M$ by means of standard cancellations. Let
 $a_{\tau\sigma}$ and $a_{\pi\theta}$ be two invertible entries of $\mathbb{F}$, corresponding to two disjoint pairs of basis elements 
 $\left([\sigma],[\tau]\right)$ and $\left([\theta],[\pi]\right)$ of $\mathbb{F}$, respectively. Then after making the standard cancellation 
 $0\rightarrow S[\sigma]\rightarrow S[\tau]\rightarrow0$ in $\mathbb{F}$, it is possible to make the 
 standard cancellation $0\rightarrow S[\theta]\rightarrow S[\pi]\rightarrow0$.
\end{lemma}
\begin{proof}
 $[\sigma]$ and $[\tau]$ are basis elements in homological degrees $j$ and $j-1$, respectively, for some $j$. Thus $a_{\tau\sigma}$ is an entry of the 
 differential matrix $\left(f_j\right)$ of $\mathbb{F}$. Similarly, $[\theta]$ and $[\pi]$ are basis elements in some homological degrees $k$ and $k-1$,
 and $a_{\pi\theta}$ is an entry of the differential matrix $\left(f_k\right)$ of $\mathbb{F}$.
 
 In order to prove the lemma, it is enough to show that after making the standard cancellation 
 $0\rightarrow S[\sigma]\rightarrow S[\tau]\rightarrow0$ in $\mathbb{F}$, the entry $a'_{\pi\theta}$ of the differential matrix $\left(f'_k\right)$ of the 
 new resolution $\mathbb{F}'$ is invertible.
 
 Given that only $\left(f_{j+1}\right)$, $\left(f_j\right)$ and $\left(f_{j-1}\right)$ are affected by the standard cancellation
 $0\rightarrow S[\sigma]\rightarrow S[\tau]\rightarrow0$, if $k\neq j-1,j,j+1$ then $a'_{\pi\theta}=a_{\pi\theta}$; that is, $a'_{\pi\theta}$ is invertible.
 Therefore, we only need to prove that $a'_{\pi\theta}$ is invertible in the following cases:\\
 $k=j$; $k=j-1$, $k=j+1$.
 
 Suposse $k=j$. Since $a_{\pi\theta}$ is invertible, $\mdeg[\pi]=\mdeg[\theta]$. Then  $a'_{\pi\theta}=0$ or  $a'_{\pi\theta}$ is invertible. Let us assume that 
 $a'_{\pi\theta}=0$.
 By Lemma 3.2 (iii), we have that $0=a'_{\pi\theta}=a_{\pi\theta}-\dfrac{a_{\pi\sigma}a_{\tau\theta}}{a_{\tau\sigma}}$. It follows that  
 $a_{\pi\theta}a_{\tau\sigma}= a_{\pi\sigma}a_{\tau\theta}$ and, since  $a_{\pi\theta}$ and $a_{\tau\sigma}$ are invertible, 
  $a_{\pi\sigma}$ and $a_{\tau\theta}$ must be invertible too. In particular, the fact that $a_{\pi\sigma}$ is invertible 
  implies that $\mdeg[\sigma]=\mdeg[\pi]$ which, combined with the hypothesis $\mdeg[\sigma]=\mdeg[\tau]$, implies that $\mdeg[\tau]=\mdeg[\pi]$.
  
  In particular, $[\tau]$ and $[\pi]$ contain the
 same dominant monomials and thus they differ in the nondominant monomials that define them. Since $[\tau]$ and $[\pi]$ appear in the same homological
 degree, they must contain exactly one nondominant generator each. Then $[\tau]$ and $[\pi]$ are of the form 
 $[\tau]=[m_{i_1},\ldots,m_{i_{j-1}},n_1]$; $[\pi]=[m_{i_1},\ldots,m_{i_{j-1}},n_2]$. Given that $\mdeg[\tau]=\mdeg[\theta]$, and the fact that
 $[\tau]$ and $[\theta]$ appear in homological degrees $j-1$ and $j$, respectively, it follows from Lemma 6.2 that $[\tau]$ is a facet of $[\theta]$. Thus
  $\theta$ must be of the form $[\theta]=[m_{i_1},\ldots, m_{i_{j-1}},n_1,n_2]$. Since $[\tau]$ is also a facet of $[\sigma]$, the same reasoning applies to
  $[\sigma]$, which means that $[\sigma]=[\theta]$, an absurd. We conclude that $a'_{\pi\theta}$ is invertible.
  
  The cases $k=j-1$ and $k=j+1$ are as in the proof of Lemma 5.5.  
 \end{proof}
\begin{theorem}
 Let $M$ be a 2-semidominant ideal. Let $\left([\sigma_1],[\tau_1]\right),\ldots,\left([\sigma_k],[\tau_k]\right)$ be $k$ pairs of basis elements of 
 $\mathbb{T}_M$, satisfying the following properties:
 \begin{enumerate}[(i)]
  \item $\left([\sigma_i],[\tau_i]\right)$ and $\left([\sigma_j],[\tau_j]\right)$ are disjoint, if $i\neq j$.
 \item $[\tau_i]$ is a facet of $[\sigma_i]$ for all $i=1,\ldots k$.
 \item $\mdeg[\sigma_i]=\mdeg[\tau_i]$ for all $i=1,\ldots k$.
 \end{enumerate}
  Then, starting with $\mathbb{T}_M$, it is possible to make the following sequence of standard cancellations:
 \[0\rightarrow S[\sigma_1]\rightarrow S[\tau_1]\rightarrow0,\quad\cdots\quad,0\rightarrow S[\sigma_k]\rightarrow S[\tau_k]\rightarrow0\]
\end{theorem} 
\begin{proof}
 Identical to the proof of Theorem 5.6.
 \end{proof}
\begin{theorem}
 Let $M$ be a 2-semidominant ideal. Let $A=\left\{([\sigma_1],[\tau_1]),\ldots,([\sigma_k],[\tau_k])\right\}$ be a family of pairs of basis elements in 
 $\mathbb{T}_M$, having the following properties:
 \begin{enumerate}[(i)]
  \item $\left([\sigma_i],[\tau_i]\right)$ and $\left([\sigma_j],[\tau_j]\right)$ are disjoint, if $i\neq j$.
 \item $[\tau_i]$ is a facet of $[\sigma_i]$ for all $i=1,\ldots k$.
 \item $\mdeg[\sigma_i]=\mdeg[\tau_i]$ for all $i=1,\ldots k$.
 \item $A$ is maximal with respect to inclusion among the sets satisfying i), ii) and iii).
 \end{enumerate}
   Then a minimal free resolution $\mathbb{F}$ of $S/M$ can be obtained from $\mathbb{T}_M$ by doing all standard cancellations 
 $0\rightarrow S[\sigma]\rightarrow S[\tau]
 \rightarrow0$, with $\left([\sigma],[\tau]\right)\in A$. In symbols, 
 \[\mathbb{T}_M=\mathbb{F}\oplus\left(\bigoplus_{\left([\sigma],[\tau]\right)\in A}0\rightarrow S[\sigma]\rightarrow S[\tau]\rightarrow0\right)\]
\end{theorem}
\begin{proof}
 By Theorem 6.4, $\mathbb{F}$ is a resolution of $S/M$. We claim that $\mathbb{F}$ is minimal. If $\mathbb{F}$ were not minimal, one of its differential 
 matrices would contain an invertible entry. That, in turn, would mean 
 that there exists a pair $([\sigma],[\tau])$ of basis elements of $\mathbb{T}_M$, such that $A\bigcup \{([\sigma],[\tau])\}$ satisfies
 conditions (i), (ii), and (iii), which contradicts (iv).
\end{proof}

We have explained that all minimal resolutions of 1-semidominant ideals, obtained from $\mathbb{T}_M$ by eliminating faces and facets of equal multidegree,
have a common basis. However, the bases of the minimal resolutions of 2-semidominant ideals, obtained in the same way,
are not unique, as the next example shows.

\begin{example}
Let $M=(x^2y^2,xz,yz)$. The only repeated multidegree is $m=x^2y^2z$, which is common to the three basis elements $[\sigma]=[x^2y^2,xz,yz]$, 
$[\tau_1]=[x^2y^2,xz]$, and $[\tau_2]=[x^2y^2,yz]$. By eliminating the pair $[\sigma]$, $[\tau_1]$ from the basis of $\mathbb{T}_M$, we obtain 
the basis of a minimal resolution of $S/M$. By eliminating the pair $[\sigma]$, $[\tau_2]$ from the basis of $\mathbb{T}_M$, we obtain a different basis 
of another minimal resolution of $S/M$.
\end{example}
\begin{theorem}(Characterization of the Scarf 2-semidominant Ideals)\\ 
 Let $M$ be a 2-semidominant ideal. Let $B=\{m:m\text{ is the multidegree of more than one basis}$\\
 $\text {element of }\mathbb{T}_M\}$.
 For each $m\in B$, let $B_m=\{[\sigma]\in \mathbb{T}_M/\text{mdeg}[\sigma]=m\}$. Then $M$ is Scarf if and only if $\#\left(B_m\right)$ is even for all 
 $m\in B$.
\end{theorem}
\begin{proof}
 Let $G=\{m_1,\ldots,m_q,n_1,n_2\}$ be the minimal generating set of $M$. Let us denote with $\mathbb{F}$ the minimal resolution of $S/M$.\\
 $(\Rightarrow)$ Let $m\in B$. Being $M$ Scarf, all elements of $B_m$ are excluded from the basis of $\mathbb{F}$, but the elements of 
 $B_m$ are eliminataded
 in pairs, making standard cancellations. It follows that $\#\left(B_m\right)$ is even.\\
 $(\Leftarrow)$ Let $m\in B$. We need to prove that no element of the basis of $\mathbb{F}$ has multidegree $m$. Given that basis elements of $\mathbb{T}_M$
 with the same multidegree contain the same dominant monomials, what distinguishes these elements is the nondominant monomials that define them.
 Thus there are at most four basis elements of multidegree $m$; namely, 
 \[[\sigma_1]=[m_{i_1},\ldots, m_{i_r}];\quad[\sigma_2]=[m_{i_1},\ldots, m_{i_r},n_1];\]
 \[[\sigma_3]=[m_{i_1},\ldots, m_{i_r},n_2];\quad [\sigma_4]=[m_{i_1},\ldots, m_{i_r},n_1,n_2]\]
  
  The fact that $\#\left(B_m\right)$ is even implies that either\\
  (i) $\#\left(B_m\right)=4$ or (ii) $\#\left(B_m\right)=2$.\\
  (i) In this case $\left([\sigma_2],[\sigma_1]\right)$, $\left([\sigma_4],[\sigma_3]\right)$ and $\mathbb{T}_M$ satisfy the hypotheses of Theorem 6.3, which 
  means that after making the standard cancellation $0\rightarrow S[\sigma_2]\rightarrow S[\sigma_1]\rightarrow0$ in $\mathbb{T}_M$, it is still possible
  to make the cancellation $0\rightarrow S[\sigma_4]\rightarrow S[\sigma_3]\rightarrow0$. Hence, the basis of $\mathbb{F} $ does not contain elements of 
  multidegree $m$.\\
  (ii) We will show that the two basis elements with multidegree $m$ are face and facet. There are exactly two pairs of basis elements that are not face 
  and facet; these pairs are $[\sigma_2],[\sigma_3]$ and $[\sigma_1],[\sigma_4]$. If we assume that $\mdeg[\sigma_2]=\mdeg[\sigma_3]=m$, then
  $n_2\mid \mdeg[\sigma_3]=\mdeg[\sigma_2]$. It follows that $\mdeg[\sigma_4]=\mdeg[\sigma_2]$ and thus $[\sigma_4],[\sigma_2]$ and 
  $[\sigma_3]$ have multidegree $m$, which is not possible because $\#\left(B_m\right)=2$.
  
  Similarly, if $\mdeg[\sigma_1]=\mdeg[\sigma_4]$, then $n_2\mid \mdeg[\sigma_4]=\mdeg[\sigma_1]$, which implies that $[\sigma_3]$,
  $[\sigma_1]$ and $[\sigma_4]$ have multidegree $m$, which is not possible. Therefore, if $\mdeg[\sigma_i]=\mdeg[\sigma_j]=m$, then 
  $[\sigma_i]$ and $[\sigma_j]$ must be face and facet. Thus they determine an invertible entry of $\mathbb{T}_M$ and it is possible to eliminate 
  $[\sigma_i]$ and $[\sigma_j]$ from the basis of $\mathbb{T}_M$ by means of a standard cancellation. This means that no element of the basis of 
  $\mathbb{F}$ has multidegree $m$.
\end{proof}
Theorem 6.7 gives a complete characterization of the Scarf 2-semidominant ideals. This characterization, however, is 
difficult to verify in practice because it requires several calculations. In order to have a good mix between theoretical and practical results, we
include two criteria to determine whether a 2-semidominant ideal is Scarf. These two tests, although weaker than the preceding theorem, are easy to 
implement in concrete cases.
\begin{corollary}
 Let $M=(m_1,\ldots,m_q,n_1,n_2)$ be 2-semidominant. If $M$ is Scarf, then\\ $n_1, n_2\mid\lcm(m_1,\ldots,m_q)$.
 \end{corollary}
 \begin{proof}
  Let $m=\mdeg[m_1,\ldots,m_q,n_1,n_2]$. Since $n_1$ is nondominant, $n_1\mid \lcm(m_1,\ldots,m_q,n_2)$, which means that $m=\mdeg[m_1,\ldots,m_q,n_2]$.
  Similarly, since $n_2$ is nondominant, we must have that $n_2\mid \lcm(m_1,\ldots,m_q,n_1)$ and this implies that $m=\mdeg[m_1,\ldots,m_q,n_1]$.
  This means that at least three basis elements of $\mathbb{T}_M$ have multidegree $m$. Now, in the proof of Theorem 6.7 we showed
  that for 2-semidominant ideals, there are at most four basis elements of $\mathbb{T}_M$ with a given multidegree. In our case, the fourth candidate is
  $[m_1,\ldots,m_q]$. If $M$ is Scarf, it follows from Theorem 6.7 that the number of basis elements of $\mathbb{T}_M$ with multidegree $m$ is even.
  Thus, we must have that $m=\mdeg[m_1,\ldots,m_q]$.
  
  The last two equations imply that $n_1\mid\lcm(m_1,\ldots,m_q)$. Similarly, $n_2\mid\lcm(m_1,\ldots,m_q)$.
 \end{proof}
 \begin{corollary}
  Let $M=(m_1,\ldots,m_q,n_1,n_2)$ be 2-semidominant. If no variable appears with the same nonzero exponent in $n_1$ and $n_2$, then $M$ is Scarf.
 \end{corollary}
\begin{proof}
 If we assume that $M$ is not Scarf, then by Theorem 6.7, there is a multidegree $m$ which is common to an odd number $k>1$ of basis elements of 
 $\mathbb{T}_M$. By the proof of Theorem 6.7, there are at most four of basis elements with multidegree $m$. They are of the form 
 $[\sigma_1]=[m_{i_1},\ldots,m_{i_r}]$; $[\sigma_2]=[m_{i_1},\ldots,m_{i_r},n_1]$; $[\sigma_3]=[m_{i_1},\ldots,m_{i_r},n_2]$; 
 $[\sigma_4]=[m_{i_1},\ldots,m_{i_r},n_1,n_2]$. Now given that $k>1$ and odd, we must have $k=3$. It is easy to verify that if exactly three of the four
 elements $[\sigma_1]$, $[\sigma_2]$, $[\sigma_3]$, $[\sigma_4]$ have multidegree $m$, these elements must be $[\sigma_2]$, $[\sigma_3]$, $[\sigma_4]$
 (in any other case, that three of these elements have multidegree $m$ would imply that the fourth one has multidegree $m$ as well).
 
 The fact that $\mdeg[\sigma_1]\neq \mdeg[\sigma_2]$ implies that $n_1\nmid\lcm(m_{i_1},\ldots,m_{i_r})$.
 In particular, there is a variable $x$ such that $x$ appears with exponent $\alpha>0$ in the factorization of $n_1$ and $x^{\alpha}\nmid
 \lcm(m_{i_1},\ldots,m_{i_r})$. On the other hand, the fact that $\mdeg[\sigma_2]=\mdeg[\sigma_3]$ implies that $x^{\alpha}
 \mid\lcm(m_{i_1},\ldots,m_{i_r},n_2)$. Therefore, $x^\alpha\mid n_2$. 
 
 Let $\beta$ be the exponent with which $x$ appears in the factorization of $n_2$. Notice that if we 
 had that $\alpha<\beta$ or $\alpha>\beta$, then we would also have that $\mdeg[\sigma_2]\neq \mdeg[\sigma_3]$. Thus $x$ appears with the same
 nonzero exponent in the factorization of $n_1$ and $n_2$, a contradiction.
\end{proof}
In the context of 2-semidominant ideals, Corollary 6.9 extends a beautiful theorem by Bayer, Peeva and Sturmfels [BPS], that states the following: If $M$ is a generic ideal,
then $M$ is Scarf.\\ Let us see how Corollaries 6.8 and 6.9 work in practice.
\begin{example}
 Let $M_1=(x^3y,y^2z,yz^4,xz^2w,x^2zw)$ and $M_2=(x^3y,y^2z,yz^4,xz^2,x^2z)$. Notice that $M_1$ is 2-semidominant, $n_1=xz^2w$ and $n_2=x^2zw$ being
 the nondominant generators. Since $w$ appears in the factorization of $n_1$ but not in the factorization of any of the dominant monomials $m_1=x^3y$,
 $m_2=y^2z$, $m_3=yz^4$, we have that $n_1\nmid\text{lcm}(m_1,m_2,m_3)$. Thus, by Corollary 6.8, we have that $M_1$ is not Scarf. Now observe that $M_2$
 is also 2-semidominant, $n_1=xz^2$ and $n_2=x^2z$ being the nondominant generators. Since neither $x$ nor $z$ appears with the same nonzero exponent in the
 factorization of $n_1$ and $n_2$, it follows from Corollary 6.9 that $M_2$ is Scarf. Incidentally, note that $M_2$ is not generic.  We chose two very 
 similar ideals $M_1$ and $M_2$ to show how sensitive monomial resolutions are.
\end{example}
\section{Conclusion}
The thread that runs through the entire study of dominant, 1-semidominant and 2-semidominant ideals is the fact that their minimal resolutions can be obtained
 eliminating pairs consisting of face and facet of equal multidegree, in arbitrary order. Of course, this principle is trivial in the case of dominant ideals
 because their Taylor resolution is already minimal and in the case of semidominant ideals, this rule is eclipsed by an even stronger fact; namely, 
 semidominant ideals are Scarf.

 In both cases, however, the principle is implicit. In order to prove that $\mathbb{T}_M$ is minimal whenever $M$ is dominant, all we have to do is show 
 that it is impossible to find a face and a facet of equal multidegree (see Theorem 4.4 ($\Leftarrow$)). Thus we do not apply the rule to $\mathbb{T}_M$ but
 we certainly study $\mathbb{T}_M$ in light of it. Similarly, the proof that semidominant ideals are Scarf is based on the fact that when we make 
 random standard cancellations involving 
 faces and facets of equal multidegree, all basis elements with a repeated multidegree are eliminated.
 
 Having understood the common theme in the study of these three classes of ideals, it is natural to wonder whether 3-semidominant ideals can be resolved in 
 the same way. Unfortunately, the answer is no, as the next example shows.
 
 With the assistance of a software system (for instance, Macaulay 2 [GS]) it is easy to verify that the 3-semidominant ideal 
 $M=(x^2y^2z^2,xw^2,yw^2,zw)$ is Scarf. Now, there are six basis elements of $\mathbb{T}_M$ with multidegree $m=x^2y^2z^2w^2$ which, therefore, are excluded
 from the basis of the minimal resolution of $S/M$. However, if we eliminate pairs of face and facet of equal multidegree as follows:\\
\[\left([x^2y^2z^2,xw^2,yw^2,zw],[x^2y^2z^2,xw^2,zw]\right)\text{ first, and }\left([x^2y^2z^2,xw^2,yw^2],[x^2y^2z^2,yw^2]\right)\text{ next,}\]
then the remaining basis elements of
multidegree $m$, $[x^2y^2z^2,yw^2,zw]$ and $[x^2y^2z^2,xw^2]$, cannot be eliminated in this way because they are not face and facet. This proves that the
basis of the minimal resolution of $S/M$ cannot be obtained eliminating pairs of face and facet of equal multidegree, at random.

It remains an open problem to determine the family of all monomial ideals the basis of whose minimal resolutions can be obtained following the rule that we
are discussing. What we know though is that the family contains more ideals than the ones we described in this paper (for instance, the 3- and 4-semidominant
ideals $M_3=(xy,xz,yz)$ and $M_4=(xz,yz,xw,yw)$ are in the family). In order to expand our knowledge of this class we need to set aside the concept of
$p$-semidominant ideal and study monomial ideals under different hypotheses.

This final section is not the right place to go deep into the study of new material, but we intend to use it as the trigger of new ideas. 
Thus we conclude this article with a theorem that may inspire similar results in the same line of reasoning.
\begin{theorem}
 Let $M$ be a monomial ideal. Let us assume that for every basis element $[\tau]$ of $\mathbb{T}_M$, which is a common facet of two faces $[\sigma_1]$ and 
 $[\sigma_2]$, such that $\mdeg[\sigma_1]=\mdeg[\sigma_2]=\mdeg[\tau]=m$, the following property holds:
 \[\text{whenever }[\tau']\neq[\tau]\text{ is a facet of }[\sigma_1]\text{ or }[\sigma_2]\text{, }\mdeg[\tau']\neq m.\]
 Then the basis of the minimal resolution of $S/M$ can be obtained from the basis of $\mathbb{T}_M$, eliminating pairs of face and facet of equal 
 multidegree in arbitrary order, until exhausting all possibilities.\qed
\end{theorem}

The proof of the theorem makes use of the foundational results of section 3. It can be shown that dominant, 1-semidominant, and 2-semidominant ideals
satisfy the hypotheses of this theorem (and so do the ideals $M_3$ and $M_4$, introduced above). This means that we could have given the minimal resolutions 
of these classes of ideals as a corollary to Theorem 7.1 but, at the expense of a minor loss of generality, we favored organization and clarity.

\bigskip

\noindent \textbf{Acknowledgements}: I want to express my gratitude to Chris Francisco and Jeff Mermin for some long conversations and helpful suggestions. 
A special thanks to my dear wife Danisa for typing many versions of this work until it reached its final form. Her support and encouragement 
made this paper possible.


\end{document}